\documentclass[11pt]{article}
\usepackage{latexsym}
\usepackage{amsfonts}

\def\<{\left\langle}
\def\>{\right\rangle}
\def\({\left(}
\def\){\right)}

\def\bea{\begin{eqnarray*}}
\def\eea{\end{eqnarray*}}

\newtheorem{main}{Theorem}

\newtheorem{thm}{Theorem}

\newtheorem{prop}{Proposition}
\newtheorem{defn}{Definition}

\newtheorem{cor}{Corollary}

\newenvironment{proof}{\medskip \noindent
{\bf Proof.}}{\hfill \rule{.5em}{1em}
\\}

\begin{document}
\sloppy
\title{Harnack Estimates for Nonlinear Heat Equations with Potentials in Geometric Flows}

\author{Hongxin Guo and Masashi Ishida}


\maketitle

\begin{abstract}
Let $M$ be a closed Riemannian manifold with a family of Riemannian metrics $g_{ij}(t)$ evolving by geometric flow $\partial_{t}g_{ij} = -2{S}_{ij}$, where $S_{ij}(t)$ is a family of smooth symmetric two-tensors on $M$. In this paper we derive differential Harnack estimates for positive solutions to the nonlinear heat equation with potential:
\begin{eqnarray*}
\frac{\partial f}{\partial t}  = {\Delta}f + \gamma (t) f\log f +aSf,
\end{eqnarray*}
where $\gamma (t)$ is a continuous function on $t$, $a$ is a constant and $S=g^{ij}S_{ij}$ is the trace of $S_{ij}$. Our Harnack estimates include many known results as special cases, and moreover lead to new Harnack inequalities for a variety geometric flows\footnote{2000 Mathematical Subject Classification: 53C44, 53C21}.
\end{abstract}

\section{Introduction}\label{intro}

Let $M$ be a closed Riemannian $n$-manifold with a one parameter family of Riemannian metrics $g(t)$ evolving by the geometric flow
\begin{eqnarray}\label{GF}
\frac{\partial}{\partial t}g_{ij} = - 2S_{ij},
\end{eqnarray}
where $S_{ij}(t)$ is a one parameter family of smooth symmetric two-tensors on $M$ and $t \in [0, T)$.

In a recent preprint \cite{G-I}, the authors studied Harnack inequalities for all positive solutions to
$$
\frac{\partial f}{\partial t} = -{\Delta}f + \gamma f\log f +aSf
$$
where $\gamma$ and $a$ are constants.
In the case where $S_{ij}=R_{ij}, \gamma=0$ and $a=1$, the above equation is Perelman's conjugate heat equation, and Harnack
estimates for all positive solutions have been studied by Cao \cite{C} and Kuang-Zhang \cite{K-Z}.

 The purpose of the current article is to study the forward nonlinear equations with potential terms under (\ref{GF}):
\begin{eqnarray}\label{forward HE}
\frac{\partial f}{\partial t}  = {\Delta}f + \gamma (t) f\log f +aSf,
\end{eqnarray}
where $\gamma (t)$ is a funtion on $t$ and $a$ is a constant. In the Ricci flow case, the consideration of this equation is motivated by expanding gradient Ricci solitons. See \cite{C-Z} for more details. In the Ricci flow, Cao-Hamilton \cite{C-H} proved various Harnack inequalities of
(\ref{forward HE}) for $\gamma(t)\equiv 0$. For general geometric flows, many people have studied Harnack inequality for the time-dependant heat equation, see for instance \cite{B-C-P, G-F, Liu}. For a positive solution $f$ of (\ref{forward HE}), let $u = - \log f$ and a direct computation tells us that $u$ satisfies
\begin{eqnarray}\label{forward HE u}
\frac{\partial u}{\partial t} = \Delta u - |\nabla u|^{2} + \gamma(t) u -aS .
\end{eqnarray}
Note that (\ref{forward HE}) and (\ref{forward HE u}) are equivalent equations.
\par
To state the main results, we introduce two quantities defined by Reto M\"uller \cite{Mu-1}.
\begin{defn}\label{h-d-de-1}
Suppose that $g(t)$ evolves by the geometric flow (\ref{GF}) and let $X=X^{i}\frac{\partial}{\partial x^{i}}$ be a vector field on $M$. One defines
\begin{eqnarray*}
{\mathcal H}(S_{ij}, {X})&=&\frac{\partial S}{\partial t} + \frac{S}{t} - 2{\nabla}_{i}SX^{i} + 2{S}^{ij}X_{i}{X}_{j}, \\
{\mathcal D}(S_{ij}, {X})&=&\frac{\partial S}{\partial t} - \Delta S - 2|S_{ij}|^{2}  + \left(4{\nabla}^{i}S_{i\ell} -2 {\nabla}_{\ell}S\right) {X}^{\ell} + 2 \left(R^{ij}-S^{ij}\right)X_iX_j
\end{eqnarray*}
where the upper indices are lifted by the metric, for instance $S^{ij}=g^{ik}g^{lj}S_{kl}$.
\end{defn}
We notice that $\mathcal H$ and $\mathcal D$ were firstly introduced by M{\"{u}}ller \cite{Mu-1} to prove the monotonicity of Perelman type reduced volume under (\ref{GF}). Later on they were used to prove entropy monotonicity and Harnack inequalities in \cite{G-H, G-I, G-P-T}. We also notice that when $M$ is static, namely when $S_{ij}=0$ one has
$$\mathcal H(0, X)=0, \quad \mathcal D(0, X)=R^{ij}X_iX_j.$$
In the Ricci flow, namely when $S_{ij}=R_{ij}$ one has
$$
\mathcal H(R_{ij}, X)=\frac{\partial R}{\partial t} + \frac{R}{t} - 2{\nabla}_{i}RX^{i} + 2{R}^{ij}X_{i}{X}_{j}, \quad \mathcal D(R_{ij}, X)=0
$$
and in this case $\mathcal H$ is nothing but Hamilton's trace Harnack quantity.
\par
For the equation (\ref{forward HE u}) in the case where $a=1$, we prove

\begin{main}\label{main-E}
Let $g(t)$ be a solution to the geometric flow (\ref{GF}) on a closed oriented smooth $n$-manifold $M$.
Assume for all $X$ and $t \in [0, T)$, it holds
\begin{eqnarray}\label{forward-assum-0}
2{\mathcal H}(S_{ij}, X) + {\mathcal D}(S_{ij}, X) \geq 0, \ S \geq 0
\end{eqnarray}
Let $u$ be a solution to
\begin{eqnarray*}
\frac{\partial u}{\partial t} = \Delta u - |\nabla u|^{2} + \gamma(t) u-S
\end{eqnarray*}
with
\begin{eqnarray}\label{func-condi}
-\frac{2}{t} \leq \gamma(t) \leq 0
\end{eqnarray}
for all time $t \in (0, T)$. Then for all $t \in (0, T)$,
\begin{eqnarray}\label{Har A}
{Q}_{S} = 2\Delta u - |\nabla u|^{2} -3{S} -2\frac{n}{t}\le 0.
\end{eqnarray}
\end{main}
Notice that in \cite {G-H}, (\ref{Har A}) was proved for $\gamma(t)\equiv 0$ under a slight different assumption.
On the other hand, notice that (\ref{func-condi}) is not satisfied for all time $t \in (0, T)$ if $\gamma(t)$ is a nonzero constant.
However, in the case where $\gamma(t) \equiv -1$, we are able to prove a similar result as follows:
\begin{main}\label{main-Ee}
Let $g(t)$ be a solution to the geometric flow (\ref{GF}) on a closed oriented smooth $n$-manifold $M$.
Assume that (\ref{forward-assum-0}) holds, namely
$2{\mathcal H}(S_{ij}, X) + {\mathcal D}(S_{ij}, X) \geq 0$ and $S \geq 0$.
Let $u$ be a solution to
\begin{eqnarray*}
\frac{\partial u}{\partial t} = \Delta u - |\nabla u|^{2} - u-S
\end{eqnarray*}
then for all time $t \in (0, T)$, the following holds:
\begin{eqnarray}\label{Har B}
{Q}_{S} = 2\Delta u - |\nabla u|^{2} -3{S} -2\frac{n}{t}\leq \frac{n}{4}.
\end{eqnarray}
\end{main}

For the equation (\ref{forward HE}) in the case where $a=0$, we shall prove
\begin{main}\label{main-F}
Suppose that $g(t)$, $t \in [0, T)$, evolves by the geometric flow (\ref{GF}) on a closed oriented smooth $n$-manifold $M$ with
\begin{eqnarray}\label{I-F-C-0}
{\mathcal I}(S_{ij}, X) := \left(R^{ij}-S^{ij}\right)X_iX_j \geq 0
\end{eqnarray}
for all $X$ and all time $t \in [0, T)$. Let $0<f<1$ be a positive solution to
\begin{eqnarray*}
\frac{\partial f}{\partial t} = {\Delta}_{}f +\gamma (t) f\log f,
\end{eqnarray*}
and $u = -\log f$. If $\gamma (t) \leq 0$ for all time $t \in [0, T)$, then
\begin{eqnarray}\label{Har C}
|\nabla u|^{2} - \frac{u}{t} \leq 0
\end{eqnarray}
holds for all time $t \in (0, T)$.
\end{main}

We notice that the above theorems in the case where $S_{ij}=R_{ij}$ imply the results proved in \cite{C-Z} as special cases.
In Theorems \ref{main-E} and \ref{main-Ee}, the assumptions are the same as stated by (\ref{forward-assum-0}).
In Theorem \ref{main-F}, the assumption is (\ref{I-F-C-0}). In the following section we will discuss the assumptions
in various geometric flows, and replace them by natural geometric assumptions in the corresponding flow. The rest of the article is devoted to proving the main theorems.

\vspace{0.2cm}
\noindent
{\bf Acknowledgements}
The first author was supported by NSFC (Grants No. 11001203 and 11171143) and Zhejiang Provincial Natural Science Foundation of China (Project No. LY13A010009). The second author was partially supported by the Grant-in-Aid for Scientific Research (C), Japan Society for the Promotion of Science, No. 25400074.

\section{Examples}\label{ex}

\noindent{\bf{(1) Static Riemannian manifold.}} In this case $S_{ij}=0$, $\mathcal H=0$ and $\mathcal D=R^{ij}X_iX_j$. Thus the assumptions
in Theorems \ref{main-E}, \ref{main-Ee} and \ref{main-F} can be replaced by $R_{ij}\ge 0$.

\noindent{\bf{(2) The Ricci flow.}} In this case $S_{ij} = {R}_{ij}$. Therefore, (\ref{forward-assum-0}) is equivalent to
\begin{eqnarray*}\label{curvature-1}
{\mathcal H}(S_{ij}, X) = \frac{\partial R}{\partial t} + \frac{R}{t} - 2{\nabla}_{i}RX^{i} + 2{R}^{ij}X_{i}X_{j} \geq 0, \ R \geq 0.
\end{eqnarray*}
It is known \cite{ha-1} that these conditions are satisfied if the initial metric $g(0)$ has weakly positive curvature operator. Hence, the assumptions in Theorems \ref{main-E} and  \ref{main-Ee} hold if $g(0)$ has weakly positive curvature operator.
Moreover, the assumption (\ref{I-F-C-0}) in Theorem \ref{main-F} is automatically satisfied. Notice that Theorem \ref{main-E} in the case where
$\gamma(t) = -2/(t + 2)$
is Theorem 1.2 in \cite{C-Z}. On the other hand, our Theorem \ref{main-Ee} is Theorem 1.1 in \cite{C-Z}.
Theorem \ref{main-F} in the case where $\gamma (t) =-1$ is nothing but Theorem 4.1 in \cite{Wu}.

\vspace{0.3cm}
\noindent{\bf{(3) List's extended Ricci flow.}} In this case, $S_{ij} = {R}_{ij}-2{\nabla}_i \psi {\nabla}_j \psi$ we have
\begin{eqnarray*}
{\mathcal D}(S_{ij}, X)= 4|\Delta \psi - {\nabla}_{X}\psi|^{2}
\end{eqnarray*}
In particular, (\ref{forward-assum-0}) is particularly satisfied if
\begin{eqnarray*}
{\mathcal H}(S_{ij}, X) \geq 0, \ R(0) \geq 2|\nabla \psi|^{2}_{t=0}.
\end{eqnarray*}
To the best our knowledge, it is still unknown whether ${\mathcal H}(S_{ij}, X) \geq 0$ is preserved by the Bernhard List's flow under a suitable assumption. The Ricci flow case is due to Hamilton \cite{ha-1} as we already mentioned. On the other hand, (\ref{I-F-C-0}) holds automatically.

\vspace{0.3cm}
\noindent{\bf{(4) M{\"{u}}ller's Ricci flow coupled with harmonic map flow.}}
In this case, $S_{ij} = {R}_{ij}-\alpha(t){\nabla}_i \phi {\nabla}_j \phi$ and moreover
${\mathcal D}(S_{ij}, X)= 2{\alpha(t)}\Big| \tau_{g} \phi - {\nabla}_{X}\phi \Big|^{2} -(\frac{\partial \alpha(t)}{\partial t})|\nabla \phi|^{2}.$
Therefore ${\mathcal D}(S_{ij}, X) \geq 0$ holds if $\alpha(t) \geq 0$ and $\frac{\partial \alpha(t)}{\partial t} \leq 0$. In this case, (\ref{forward-assum-0})  is particularly satisfied if
\begin{eqnarray*}
{\mathcal H}(S_{ij}, X) \geq 0, \ R(0) \geq \alpha(0)|\nabla \phi|^{2}_{t=0}.
\end{eqnarray*}
As in the case of List's flow, it is unknown whether ${\mathcal H}(S_{ij}, X) \geq 0$ is preserved by M{\"{u}}ller's flow under a suitable assumption.
As in other examples (\ref{I-F-C-0}) holds automatically.

\section{General evolution equations}

In this section, we shall prove general evolution equations of Harnack quantities under the geometric flow, which are useful to prove the main results. See Theorems \ref{for-thm-2} and \ref{for-thm-3} stated below. In the Ricci flow case, these general evolution equations are firstly proved by Cao and Hamilton \cite{C-H}. Theorems \ref{for-thm-2} and \ref{for-thm-3} can be seen as generalizations of Lemma 2.1 and Lemma 3.1 in \cite{C-H} respectively.

\subsection{Case of $u = - \log f$}\label{for-u-1}

Let $M$ be a closed Riemannian manifold with a Riemannian metric $g_{ij}(t)$ evolving by a geometric flow $\partial_{t}g_{ij} = -2{S}_{ij}$. Let $f$ be a positive solution of the following equation:
\begin{eqnarray}\label{nonlinear-heat-c}
\frac{\partial f}{\partial t} = {\Delta}_{}f + \gamma(t) f\log f - cSf,
\end{eqnarray}
where $c$ is a constant and $\gamma (t)$ is a function depneds on $t$. Let $u = - \log f$. A direct computation tells us that $u$ satisfies
\begin{eqnarray}\label{for-u-differ0}
\frac{\partial u}{\partial t} = \Delta u - |\nabla u|^{2} + cS + \gamma(t) u.
\end{eqnarray}
We introduce
\begin{defn}\label{Def-HJ}
Suppose that $g(t)$ evolves by (\ref{GF}) and let $S$ be the trace of $S_{ij}$. Let $X=X^{i}\frac{\partial}{\partial x^{i}}$ be a vector field on $M$. And $a, \alpha$ and $\beta$ are constants. Then, one defines
\begin{eqnarray*}
{\mathcal D}_{(a, \alpha, \beta)}(S_{ij}, X) &=& a \Big(\frac{\partial S}{\partial t} - \Delta S - 2|S_{ij}|^{2} \Big) - \alpha \Big(2{\nabla}^{i}S_{i\ell} - {\nabla}_{\ell}S \Big){X}^{\ell} \\
&+& 2{\beta} (R^{ij} -S^{ij}){X}_{i}{X}_{j}.
\end{eqnarray*}
\end{defn}
Then we prove
\begin{prop}\label{for-lem-key-1}
Let $g(t)$ be a solution to the geometric flow (\ref{GF}) and $u$ satisfies (\ref{for-u-differ0}). Let
\begin{eqnarray*}
{Q}_{S} = \alpha \Delta u - \beta |\nabla u|^{2} + a{S} - b\frac{u}{t}-d\frac{n}{t},
\end{eqnarray*}
where $\alpha, \beta, a, b$ and $d$ are constants. Then $Q_{S}$ satisfies
\begin{eqnarray*}
\frac{\partial Q_{S}}{\partial t}
&=& {\Delta}Q_{S} - 2{\nabla}^{i}Q_{S}{\nabla}_{i}u + 2(a-\beta c){\nabla}^{i}S{\nabla}_{i}u-2(\alpha-\beta)|\nabla \nabla u|^{2} \\
&-&2{\alpha}R^{ij}{\nabla}_{i}u{\nabla}_{j}u + 2{\alpha}S^{ij}{\nabla}_{i}{\nabla}_{j}u + {\alpha}c{\Delta}S -\frac{b}{t}|\nabla u|^{2} - \frac{b}{t}cS + \frac{b}{t^{2}}u + d\frac{n}{t^{2}} \\
&+& 2a|S_{ij}|^{2} + {\mathcal D}_{(a, \alpha, \beta)}(S_{ij}, -\nabla u) + \alpha \gamma(t) \Delta u - 2{\beta}\gamma(t) |\nabla u|^{2} - b\frac{\gamma(t)}{t}u.
\end{eqnarray*}
\end{prop}
\begin{proof}
First of all, notice that we have the following evolution equations, which follow from standard computations:
\begin{eqnarray*}
\frac{\partial}{\partial t} (\Delta u) &=& 2{S}^{ij}{\nabla}_{i}{\nabla}_{j}u + \Delta(\frac{\partial u}{\partial t}) - g^{ij}\Big(\frac{\partial}{\partial t}\Gamma^{k}_{ij} \Big)\nabla_{k}u, \\
\frac{\partial}{\partial t}(|\nabla u|^{2}) &=& 2{S}^{ij}{\nabla}_{i}u{\nabla}_{j}u + 2 \nabla^{i}(\frac{\partial u}{\partial t}) \nabla_{i} u.
\end{eqnarray*}
On the other hand, we also get the following by standard computations:
\begin{eqnarray*}
g^{ij}\Big(\frac{\partial}{\partial t}\Gamma^{k}_{ij} \Big) &=& -g^{k \ell} (2\nabla^{i}S_{i\ell} - \nabla_{\ell}S).
\end{eqnarray*}
By using these formulas and (\ref{for-u-differ0}), we obtain the following:
\begin{eqnarray*}
\frac{\partial Q_{S}}{\partial t} &=& \alpha \frac{\partial}{\partial t} (\Delta u) - \beta \frac{\partial}{\partial t}(|\nabla u|^{2}) + a\frac{\partial S}{\partial t}- \frac{b}{t}\frac{\partial u}{\partial t} + \frac{b}{t^{2}}u + d\frac{n}{t^{2}} \\
&=& \alpha \Big(2{S}^{ij}{\nabla}_{i}{\nabla}_{j}u + \Delta(\frac{\partial u}{\partial t}) - g^{ij}\Big(\frac{\partial}{\partial t}\Gamma^{k}_{ij} \Big)\nabla_{k}u \Big) \\
&-& \beta \Big(2{S}^{ij}{\nabla}_{i}u{\nabla}_{j}u + 2 \nabla^{i}(\frac{\partial u}{\partial t}) \nabla_{i} u \Big) + a\frac{\partial S}{\partial t} - \frac{b}{t}\frac{\partial u}{\partial t} + \frac{b}{t^{2}}u + d\frac{n}{t^{2}} \\
&=& \alpha \Big(2{S}^{ij}{\nabla}_{i}{\nabla}_{j}u + \Delta(\Delta u - |\nabla u|^{2} + cS + \gamma(t) u)
+ g^{k \ell} (2\nabla^{i}S_{i\ell} - \nabla_{\ell}S)\nabla_{k}u \Big) \\
&-& \beta \Big(2{S}^{ij}{\nabla}_{i}u{\nabla}_{j}u + 2 \nabla^{i}(\Delta u - |\nabla u|^{2} + cS + \gamma(t) u) \nabla_{i} u \Big) + a\frac{\partial S}{\partial t} \\
&-& \frac{b}{t}(\Delta u - |\nabla u|^{2} + cS + \gamma(t) u) + \frac{b}{t^{2}}u + d\frac{n}{t^{2}} \\
&=& 2 \alpha{S}^{ij}{\nabla}_{i}{\nabla}_{j}u + \alpha \Delta(\Delta u) - \alpha \Delta(|\nabla u|^{2}) + \alpha c \Delta S+  \alpha(2\nabla^{i}S_{i\ell} - \nabla_{\ell}S)\nabla^{\ell}u \\
&-& 2\beta {S}^{ij}{\nabla}_{i}u{\nabla}_{j}u - 2{\beta}\nabla^{i}(\Delta u)\nabla_{i}u + 2{\beta}\nabla^{i}(|\nabla u|^{2})\nabla_{i}u-2{\beta} c\nabla^{i}S \nabla_{i}u \\
&+& \frac{b}{t}|\nabla u|^{2} - \frac{b}{t}cS + \frac{b}{t^{2}}u + d\frac{n}{t^{2}} + a\frac{\partial S}{\partial t} - \frac{b}{t}\Delta u + \alpha \gamma(t) \Delta u - 2{\beta}\gamma(t) |\nabla u|^{2} - b\frac{\gamma(t)}{t}u.
\end{eqnarray*}
On the other hand, we also have the following by the definition of $Q_{S}$:
\begin{eqnarray*}\label{Def-H-22}
{\Delta}Q_{S} &=& \alpha \Delta(\Delta u) - \beta \Delta(|\nabla u|^{2}) + a\Delta{S} - \frac{b}{t}\Delta u. \\
{\nabla}^{i}Q_{S} &=& \alpha \nabla^{i}(\Delta u) - \beta \nabla^{i}(|\nabla u|^{2}) + a\nabla^{i}{S} - \frac{b}{t}\nabla^{i}u
\end{eqnarray*}
Therefore we get
\begin{eqnarray*}
{\Delta}Q_{S}-2{\nabla}^{i}Q_{S}{\nabla}_{i}u &=& \alpha \Delta(\Delta u) - \beta \Delta(|\nabla u|^{2}) + a\Delta{S} - \frac{b}{t}\Delta u \\
&-& 2{\alpha}{\nabla}^{i}(\Delta u){\nabla}_{i}u + 2\beta{\nabla}^{i}(|\nabla u|^{2}){\nabla}_{i}u - 2a{\nabla}^{i}S{\nabla}_{i}u + \frac{2b}{t}|\nabla u|^{2}.
\end{eqnarray*}
By using this, we are able to obtain
\begin{eqnarray*}
\frac{\partial Q_{S}}{\partial t} &=& {\Delta}Q_{S} - 2{\nabla}^{i}Q{\nabla}_{i}u + 2 \alpha{S}^{ij}{\nabla}_{i}{\nabla}_{j}u - 2\beta {S}^{ij}{\nabla}_{i}u{\nabla}_{j}u - (\alpha - \beta) \Delta (|\nabla u|^{2}) \\
&+& (\alpha c-a)\Delta S + \alpha (2{\nabla}^{i}S_{i \ell} - {\nabla}_\ell S){\nabla}^{\ell}u + 2(\alpha-\beta){\nabla}^{i}(\Delta u){\nabla}_{i}u  \\
&+& 2(a-\beta c){\nabla}^{i}S{\nabla}_{i}u -\frac{b}{t}|\nabla u|^{2} + a\frac{\partial S}{\partial t} - \frac{b}{t}cS + \frac{b}{t^{2}}u + d\frac{n}{t^{2}} \\
&+& \alpha \gamma(t) \Delta u - 2{\beta}\gamma(t) |\nabla u|^{2} - b \frac{\gamma(t)}{t}u.
\end{eqnarray*}
On the other hand, we also have the following Bochner-Weitzenbock type formula:
\begin{eqnarray*}
\Delta (|\nabla u|^{2}) = 2|\nabla \nabla u|^{2} + 2{\nabla}^{i}(\Delta u){\nabla}_{i}u + 2{R}^{ij}{\nabla}_{i}u{\nabla}_{j}u.
\end{eqnarray*}
By using this formula, we get
\begin{eqnarray*}
\frac{\partial Q_{S}}{\partial t} &=& {\Delta}Q_{S} - 2{\nabla}^{i}Q_{S}{\nabla}_{i}u + 2(a-\beta c){\nabla}^{i}S{\nabla}_{i}u-2(\alpha-\beta)|\nabla \nabla u|^{2} \\
&-&2{\alpha}R^{ij}{\nabla}_{i}u{\nabla}_{j}u + 2{\alpha}S^{ij}{\nabla}_{i}{\nabla}_{j}u + {\alpha}c{\Delta}S -\frac{b}{t}|\nabla u|^{2} - \frac{b}{t}cS + \frac{b}{t^{2}}u + d\frac{n}{t^{2}} \\
&+& 2a|S_{ij}|^{2} + a \Big(\frac{\partial S}{\partial t} -{\Delta}S - 2|S_{ij}|^{2} \Big) + \alpha(2{\nabla}^{i}S_{i \ell} - {\nabla}_\ell S){\nabla}^{\ell}u \\
&+& 2{\beta}(R^{ij} - S^{ij}){\nabla}_{i}u{\nabla}_{j}u + \alpha \gamma(t) \Delta u - 2{\beta}\gamma(t) |\nabla u|^{2} - b \frac{\gamma(t)}{t}u \\
&=& {\Delta}Q_{S} - 2{\nabla}^{i}Q_{S}{\nabla}_{i}u + 2(a-\beta c){\nabla}^{i}S{\nabla}_{i}u-2(\alpha-\beta)|\nabla \nabla u|^{2} \\
&-&2{\alpha}R^{ij}{\nabla}_{i}u{\nabla}_{j}u + 2{\alpha}S^{ij}{\nabla}_{i}{\nabla}_{j}u + {\alpha}c{\Delta}S -\frac{b}{t}|\nabla u|^{2} - \frac{b}{t}cS + \frac{b}{t^{2}}u + d\frac{n}{t^{2}} \\
&+& 2a|S_{ij}|^{2} + {\mathcal D}_{(a, \alpha, \beta)}(S_{ij}, -\nabla u) + \alpha \gamma(t) \Delta u - 2{\beta}\gamma(t) |\nabla u|^{2} - b \frac{\gamma(t)}{t}u,
\end{eqnarray*}
where we used Definition \ref{Def-HJ}.
\end{proof}


\begin{thm}\label{for-thm-2}
Suppose that $\alpha \not= 0$ and $\alpha \not= \beta$. Then, the evolution equation in Proposition \ref{for-lem-key-1} can be rewritten as follows:
\begin{eqnarray*}\label{Def-H}
\frac{\partial Q_{S}}{\partial t} &=& \Delta{Q_{S}} -2{\nabla}^{i}Q_{S}{\nabla}_{i}u - 2(\alpha - \beta)\Big|\nabla_{i}\nabla_{j}u - \frac{\alpha}{2(\alpha-\beta)}S_{ij}-\frac{\lambda}{2t}g_{ij} \Big|^{2} \\
&+& 2(a-\beta c){\nabla}^{i}u{\nabla}_{i}S - \frac{2(\alpha -\beta)}{\alpha}\frac{\lambda}{t}Q_{S} + \frac{(\alpha -\beta)n{\lambda}^{2}}{2t^{2}}- \Big(b + \frac{2(\alpha-\beta)\lambda {\beta}}{\alpha} \Big)\frac{|\nabla u|^{2}}{t} \\
&+& \Big(2a + \frac{\alpha^{2}}{2(\alpha-\beta)} \Big)|S_{ij}|^{2} + \Big(\alpha{\lambda} - bc + \frac{2(\alpha-\beta)\lambda a}{\alpha} \Big)\frac{S}{t} + \Big(1-\frac{2(\alpha-\beta)\lambda}{\alpha} \Big)\frac{b}{t^{2}}u \\
&+& \Big(1-\frac{2(\alpha-\beta)\lambda}{\alpha} \Big)\frac{d}{t^{2}}n + \alpha c {\Delta}S - 2{\alpha}R^{ij}{\nabla}_{i}u{\nabla}_{j}u + {\mathcal D}_{(a, \alpha, \beta)}(S_{ij}, -{\nabla}u) \\
& + & \alpha \gamma(t) \Delta u - 2{\beta}\gamma(t) |\nabla u|^{2} - b\frac{\gamma(t)}{t}u,
\end{eqnarray*}
where $\lambda$ is a constant.
\end{thm}
\begin{proof}
First of all, notice that a direct computation implies
\begin{eqnarray*}
&-&2(\alpha-\beta)\Big|\nabla_{i}\nabla_{j}u - \frac{\alpha}{2(\alpha-\beta)}S_{ij}-\frac{\lambda}{2t}g_{ij} \Big|^{2} =  -2(\alpha-\beta)|\nabla \nabla u|^{2} + 2{\alpha}S^{ij}{\nabla}_{i}{\nabla}_{j}u \\
&+& 2(\alpha-\beta)\frac{\lambda}{t} \Delta u - \frac{\lambda \alpha}{t}S - \frac{\alpha^{2}}{2(\alpha-\beta)} |S_{ij}|^{2} - \frac{(\alpha - \beta)\lambda^{2} n}{2t^{2}}.
\end{eqnarray*}
Therefore we get the following:
\begin{eqnarray*}
&-&2(\alpha-\beta)|\nabla \nabla u|^{2} + 2{\alpha}S^{ij}{\nabla}_{i}{\nabla}_{j}u + 2a|S_{ij}|^{2} \\
&=& -2(\alpha-\beta)\Big|\nabla_{i}\nabla_{j}u - \frac{\alpha}{2(\alpha-\beta)}S_{ij}-\frac{\lambda}{2t}g_{ij} \Big|^{2} - 2(\alpha-\beta)\frac{\lambda}{t} \Big(\Delta u - \frac{\alpha S}{2(\alpha-\beta)} \Big) \\
&+& \frac{(\alpha - \beta)\lambda^{2} n}{2t^{2}} + \Big(2a + \frac{\alpha^{2}}{2(\alpha-\beta)} \Big)|S_{ij}|^{2}.
\end{eqnarray*}
By this and Lemma \ref{for-lem-key-1}, we obtain
\begin{eqnarray*}
\frac{\partial Q_{S}}{\partial t}
&=& {\Delta}Q_{S} - 2{\nabla}^{i}Q_{S}{\nabla}_{i}u - 2(\alpha-\beta)\Big|\nabla_{i}\nabla_{j}u - \frac{\alpha}{2(\alpha-\beta)}S_{ij} - \frac{\lambda}{2t}g_{ij} \Big|^{2} \\
&+& 2(a-\beta c){\nabla}^{i}S{\nabla}_{i}u - 2(\alpha-\beta)\frac{\lambda}{t} \Big(\Delta u - \frac{\alpha S}{2(\alpha-\beta)} \Big) + \frac{(\alpha - \beta)\lambda^{2} n}{2t^{2}} \\
&+& \Big(2a + \frac{\alpha^{2}}{2(\alpha-\beta)} \Big)|S_{ij}|^{2} + \alpha c {\Delta}S - 2\alpha{R}^{ij}{\nabla}_{i}u{\nabla}_{j}u - \frac{b}{t}|\nabla u|^{2} - \frac{b}{t}cS \\
&+& \frac{b}{t^{2}}u + d\frac{n}{t^{2}} + {\mathcal D}_{(a, \alpha, \beta)}(S_{ij}, -\nabla u) + \alpha \gamma(t) \Delta u - 2{\beta}\gamma(t) |\nabla u|^{2} - b\frac{\gamma(t)}{t}u.
\end{eqnarray*}
On the other hand, we also get the following by using the definition of $Q_{S}$:
\begin{eqnarray*}
&-&2(\alpha-\beta)\frac{\lambda}{t} \Big(\Delta u - \frac{\alpha S}{2(\alpha-\beta)} \Big)- \frac{b}{t}|\nabla u|^{2} - \frac{b}{t}cS + \frac{b}{t^{2}}u + d\frac{n}{t^{2}} \\
&=& - \frac{2(\alpha - \beta)}{\alpha}\frac{\lambda}{t}Q_{S} - \Big(b + \frac{2(\alpha -\beta)\lambda \beta}{\alpha} \Big)\frac{|\nabla u|^{2}}{t} + \Big(1- \frac{2(\alpha- \beta)\lambda}{\alpha} \Big)\frac{b}{t^{2}}u \\
&+& \Big(\alpha{\lambda} -bc + \frac{2(\alpha-\beta)\lambda a}{\alpha} \Big)\frac{S}{t} + \Big(1-\frac{2(\alpha-\beta)\lambda}{\alpha}  \Big)\frac{d}{t^{2}}n.
\end{eqnarray*}
Using this equation, we get the claim.
\end{proof}

As a special case, we obtain the following result:
\begin{cor}\label{cor-D}
Let $g(t)$ be a solution to the geometric flow (\ref{GF}) and $u$ satisfies
\begin{eqnarray*}
\frac{\partial u}{\partial t} = \Delta u - |\nabla u|^{2} -(a+4)S + \gamma(t) u.
\end{eqnarray*}
Let
\begin{eqnarray*}
{Q}_{S} = 2\Delta u - |\nabla u|^{2} +a{S} -d\frac{n}{t},
\end{eqnarray*}
where $a$ and $d$ are constants. Then $Q_{S}$ satisfies
\begin{eqnarray*}
\frac{\partial Q_{S}}{\partial t} &=& \Delta{Q_{S}} -2{\nabla}^{i}Q_{S}{\nabla}_{i}u - 2\Big|\nabla_{i}\nabla_{j}u - S_{ij}-\frac{1}{t}g_{ij} \Big|^{2} - \Big(\frac{2}{t} -\gamma(t) \Big)Q_{S} \\
&+& \Big(-\frac{2}{t} -\gamma(t) \Big)|\nabla u|^{2} - a \gamma(t)S + 2(a+2){\mathcal H}(S_{ij}, -{\nabla}u) + \frac{n}{t^{2}}(2-d) + d{\gamma}(t)\frac{n}{t} \\
&-& \Big((a+4)\frac{\partial S}{\partial t} -2|S_{ij}|^{2} + (3a+8)\Delta S \Big) + 2\Big( 2{\nabla}^{i}S_{i \ell} - {\nabla}_{\ell}S \Big){\nabla}^{\ell}u \\
&-& 2 \Big(R^{ij} + (2a+5)S^{ij} \Big){\nabla}_{i}u{\nabla}_{j}u
\end{eqnarray*}
\end{cor}

\begin{proof}
By Theorem \ref{for-thm-2} in the case where $\alpha = 2$, $\beta = 1$, $b = 0$, $c=-a-4$, $\lambda = 2$, we obtain
\begin{eqnarray*}
\frac{\partial Q_{S}}{\partial t} &=& \Delta{Q_{S}} -2{\nabla}^{i}Q_{S}{\nabla}_{i}u - 2\Big|\nabla_{i}\nabla_{j}u - S_{ij}-\frac{1}{t}g_{ij} \Big|^{2} - \frac{2}{t}Q_{S} - \frac{2}{t}|\nabla u|^{2} -2(a+4)\Delta S \\
&+& 2(1+a)|S_{ij}|^{2} + 2(a+2) \Big(\frac{\partial S}{\partial t} + \frac{S}{t} + 2{\nabla}^{i}S{\nabla}_{i}u + 2S^{ij}{\nabla}_{i}u{\nabla}_{j}u \Big) \\
&-&2(a+2)\frac{\partial S}{\partial t}-4(a+2)S^{ij}{\nabla}_{i}u{\nabla}_{j}u-4R^{ij}{\nabla}_{i}u{\nabla}_{j}u + \frac{n}{t^{2}}(2-d) + {\mathcal D}_{(a, 2, 1)}(S_{ij}, -\nabla u) \\
&+& 2\gamma (t) (\Delta u - |\nabla u|^{2}) \\
&=& \Delta{Q_{S}} -2{\nabla}^{i}Q_{S}{\nabla}_{i}u - 2\Big|\nabla_{i}\nabla_{j}u - S_{ij}-\frac{1}{t}g_{ij} \Big|^{2} - \frac{2}{t}Q_{S} - \frac{2}{t}|\nabla u|^{2} -2(a+4)\Delta S \\
&+& 2(1+a)|S_{ij}|^{2} + 2(a+2){\mathcal H}(S_{ij}, -{\nabla}u)-2(a+2)\frac{\partial S}{\partial t}-4(a+2)S^{ij}{\nabla}_{i}u{\nabla}_{j}u \\
&-&4R^{ij}{\nabla}_{i}u{\nabla}_{j}u + \frac{n}{t^{2}}(2-d) + {\mathcal D}_{(a, 2, 1)}(S_{ij}, -\nabla u) + 2\gamma (t) (\Delta u - |\nabla u|^{2}).
\end{eqnarray*}
Since we have
\begin{eqnarray*}
{\mathcal D}_{(a, 2, 1)}(S_{ij}, -\nabla u) = a \Big(\frac{\partial S}{\partial t} - \Delta S - 2|S_{ij}|^{2} \Big) + 2 \Big(2{\nabla}^{i}S_{i\ell} - {\nabla}_{\ell}S \Big){\nabla}^{\ell}u + 2(R^{ij} -S^{ij}){\nabla}_{i}u{\nabla}_{j}u,
\end{eqnarray*}
we get
\begin{eqnarray*}
\frac{\partial Q_{S}}{\partial t} &=& \Delta{Q_{S}} -2{\nabla}^{i}Q_{S}{\nabla}_{i}u - 2\Big|\nabla_{i}\nabla_{j}u - S_{ij}-\frac{1}{t}g_{ij} \Big|^{2} - \frac{2}{t}Q_{S} - \frac{2}{t}|\nabla u|^{2} \\
&+& 2(a+2){\mathcal H}(S_{ij}, -{\nabla}u)-2 \Big(R^{ij} +(2a+5)S^{ij}\Big) {\nabla}_{i}u{\nabla}_{j}u + 2 \Big(2{\nabla}^{i}S_{i\ell} - {\nabla}_{\ell}S \Big){\nabla}^{\ell}u \\
&-& \Big((a+4)\frac{\partial S}{\partial t} -2|S_{ij}|^{2} + (3a+8)\Delta S \Big) + \frac{n}{t^{2}}(2-d) + 2\gamma (t) (\Delta u - |\nabla u|^{2}).
\end{eqnarray*}
On the other hand, we also get the following by a direct computation:
\begin{eqnarray*}
2\gamma (t) (\Delta u - |\nabla u|^{2}) = \gamma (t){Q}_{S} - \gamma (t) |\nabla u|^{2} - a \gamma (t) S + d \gamma (t) \frac{n}{t}.
\end{eqnarray*}
Using this, we obtain the desired result.
\end{proof}

\subsection{Case of $v = - \log f -\frac{n}{2}\log(4{\pi}t)$}

As in Subsection \ref{for-u-1}, let $f$ be a positive solution of (\ref{nonlinear-heat-c}).
Let $v = - \log f-\frac{n}{2}\log(4{\pi}t)$. A direct computation tells us that $v$ satsifies
\begin{eqnarray}\label{for-u-differ}
\frac{\partial v}{\partial t} = \Delta v - |\nabla v|^{2} + cS -\frac{n}{2t} + \gamma(t) \Big(v +\frac{n}{2}\log(4{\pi}t) \Big).
\end{eqnarray}
\begin{thm}\label{for-thm-3}
Let $g(t)$ be a solution to the geometric flow (\ref{GF}) and $u$ satisfies (\ref{for-u-differ}). Let
\begin{eqnarray*}
{R}_{S} = \alpha \Delta v - \beta |\nabla v|^{2} + a{S} - b\frac{v}{t}-d\frac{n}{t},
\end{eqnarray*}
where $\alpha, \beta, a, b$ and $d$ are constants. Assume that $\alpha \not= 0$ and $\alpha \not= \beta$. Then $R_{S}$ satisfies
\begin{eqnarray*}\label{Def-H}
\frac{\partial R_{S}}{\partial t} &=& \Delta{R_{S}} -2{\nabla}^{i}R_{S}{\nabla}_{i}u - 2(\alpha - \beta)\Big|\nabla_{i}\nabla_{j}u - \frac{\alpha}{2(\alpha-\beta)}S_{ij}-\frac{\lambda}{2t}g_{ij} \Big|^{2} \\
&+& 2(a-\beta c){\nabla}^{i}v{\nabla}_{i}S - \frac{2(\alpha -\beta)}{\alpha}\frac{\lambda}{t}R_{S} + \frac{(\alpha -\beta)n{\lambda}^{2}}{2t^{2}}- \Big(b + \frac{2(\alpha-\beta)\lambda {\beta}}{\alpha} \Big)\frac{|\nabla v|^{2}}{t} \\
&+& \Big(2a + \frac{\alpha^{2}}{2(\alpha-\beta)} \Big)|S_{ij}|^{2} + \Big(\alpha{\lambda} - bc + \frac{2(\alpha-\beta)\lambda a}{\alpha} \Big)\frac{S}{t} + \Big(1-\frac{2(\alpha-\beta)\lambda}{\alpha} \Big)\frac{b}{t^{2}}v \\
&+& \Big(1-\frac{2(\alpha-\beta)\lambda}{\alpha} \Big)\frac{d}{t^{2}}n + \alpha c {\Delta}S - 2{\alpha}R^{ij}{\nabla}_{i}v{\nabla}_{j}v + {\mathcal D}_{(a, \alpha, \beta)}(S_{ij}, -{\nabla}v) \\
& + & \alpha \gamma(t) \Delta v - 2{\beta}\gamma(t) |\nabla v|^{2} - b\frac{\gamma(t)}{t}\Big( v +\frac{n}{2}\log(4{\pi}t) \Big) + \frac{bn}{2t^{2}},
\end{eqnarray*}
where $\lambda$ is a constant.
\end{thm}

\begin{proof}
The idea of the proof is similar to that of Theorem \ref{for-thm-2}. In fact, notice that we have $v = u - \frac{n}{2}\log(4{\pi}t)$. Hence $\nabla u = \nabla v$ and $\Delta u = \Delta v$ hold. Moreover,
$$
R_{S} =Q_{S} + \frac{bn}{2t}\log(4{\pi}t).
$$
Then Theorem \ref{for-thm-2} and direct computations imply the desired result.
\end{proof}

As a special case, we get
\begin{cor}\label{cor-Da-f}
Let $g(t)$ be a solution to the geometric flow (\ref{GF}) and $v$ satisfies
\begin{eqnarray*}
\frac{\partial v}{\partial t} = \Delta v - |\nabla v|^{2} -(a+4)S -\frac{n}{2t} + \gamma(t)\Big(v + \frac{n}{2}\log(4{\pi}t)\Big).
\end{eqnarray*}
Let
\begin{eqnarray*}
{R}_{S} = 2\Delta v - |\nabla v|^{2} +a{S} -d\frac{n}{t},
\end{eqnarray*}
where $a$ and $d$ are constants. Then $R_{S}$ satisfies
\begin{eqnarray*}
\frac{\partial R_{S}}{\partial t} &=& \Delta{R_{S}} -2{\nabla}^{i}R_{S}{\nabla}_{i}v - 2\Big|\nabla_{i}\nabla_{j}v - S_{ij}-\frac{1}{t}g_{ij} \Big|^{2} - \Big(\frac{2}{t} -\gamma(t) \Big)R_{S} \\
&+& \Big(-\frac{2}{t} -\gamma(t) \Big)|\nabla v|^{2} - a \gamma(t)S + 2(a+2){\mathcal H}(S_{ij}, -{\nabla}v) + \frac{n}{t^{2}}(2-d) + d{\gamma}(t)\frac{n}{t} \\
&-& \Big((a+4)\frac{\partial S}{\partial t} -2|S_{ij}|^{2} + (3a+8)\Delta S \Big) + 2\Big( 2{\nabla}^{i}S_{i \ell} - {\nabla}_{\ell}S \Big){\nabla}^{\ell}v \\
&-& 2 \Big(R^{ij} + (2a+5)S^{ij} \Big){\nabla}_{i}v{\nabla}_{j}v
\end{eqnarray*}
\end{cor}
\begin{proof}
The idea of proof is similar to that of Corollary \ref{cor-D}. Use Theorem \ref{for-thm-3} in the case where $\alpha = 2$, $\beta = 1$, $b = 0$, $c=-a-4$, $\lambda = 2$.
\end{proof}

\section{Proof of Theorem \ref{main-E}}\label{E-proof}

By Corollary \ref{cor-D} in the case where $a=-3$, we obtain
\begin{eqnarray*}
\frac{\partial Q_{S}}{\partial t} &=& \Delta{Q_{S}} -2{\nabla}^{i}Q_{S}{\nabla}_{i}u - 2\Big|\nabla_{i}\nabla_{j}u - S_{ij}-\frac{1}{t}g_{ij} \Big|^{2} - \Big(\frac{2}{t} -\gamma(t) \Big)Q_{S} \\
&+& \Big(-\frac{2}{t} -\gamma(t) \Big)|\nabla u|^{2} + 3 \gamma(t)S +\frac{n}{t^{2}}(2-d) + d{\gamma}(t)\frac{n}{t} - \Big(2{\mathcal H}(S_{ij}, -{\nabla}u) + {\mathcal D}(S_{ij}, -{\nabla}u) \Big) \\
&\leq & \Delta{Q_{S}} -2{\nabla}^{i}Q_{S}{\nabla}_{i}u - \Big(\frac{2}{t} -\gamma(t) \Big)Q_{S} + \Big(-\frac{2}{t} -\gamma(t) \Big)|\nabla u|^{2} + 3 \gamma(t)S \\
&+& \frac{n}{t^{2}}(2-d) + d{\gamma}(t)\frac{n}{t} - \Big(2{\mathcal H}(S_{ij}, -{\nabla}u) + {\mathcal D}(S_{ij}, -{\nabla}u) \Big).
\end{eqnarray*}
Now we assume that $d \geq 2$ holds. Moreover, by the assumption of Theorem \ref{main-E}, we also get
\begin{eqnarray*}
2{\mathcal H}(S_{ij}, -{\nabla}u) + {\mathcal D}(S_{ij}, -{\nabla}u) \geq 0, \ S \geq 0, \
-\frac{2}{t} \leq \gamma(t) \leq 0.
\end{eqnarray*}
Therefore we are able to obtain
\begin{eqnarray*}
\frac{\partial Q_{S}}{\partial t} &\leq & \Delta{Q_{S}} -2{\nabla}^{i}Q_{S}{\nabla}_{i}u - \Big(\frac{2}{t} -\gamma(t) \Big)Q_{S}.
\end{eqnarray*}
Notice that
$$
Q_{S} < 0
$$
holds for $t$ small enough which depends on $d$. By using the maximal principle, we get the desired result. \par
Similarly, we get the following by using Corollary \ref{cor-Da-f}:
\begin{thm}\label{main-f-thm}
Let $g(t)$ be a solution to the geometric flow (\ref{GF}) on a closed oriented smooth $n$-manifold $M$satisfying
\begin{eqnarray*}
2{\mathcal H}(S_{ij}, X) + {\mathcal D}(S_{ij}, X) \geq 0, \ S \geq 0.
\end{eqnarray*}
hold for all vector fields $X$ and all time $t \in [0, T)$ for which the flow exists. Let $v$ satisfies
\begin{eqnarray*}
\frac{\partial v}{\partial t} = \Delta v - |\nabla v|^{2} + cS -\frac{n}{2t} + \gamma(t) \Big(v +\frac{n}{2}\log(4{\pi}t) \Big).
\end{eqnarray*}
and assume that $\gamma(t)$ satsisfies
\begin{eqnarray*}
-\frac{2}{t} \leq \gamma(t) \leq 0
\end{eqnarray*}
for for all time $t \in (0, T)$. Let
\begin{eqnarray*}
{R}_{S} = 2\Delta v - |\nabla v|^{2} -3{S} -d\frac{n}{t},
\end{eqnarray*}
where $d \geq 2$ is a constant. Then for all time $t \in (0, T)$,
$$
{R}_{S} \leq 0
$$
holds.
\end{thm}

\section{Proof of Theorem \ref{main-Ee}}\label{Ee-proof}

By Corollary \ref{cor-D} in the case where $a=-3$ and $\gamma (t) \equiv \gamma$, where $\gamma$ is a constant, we obtain
\begin{eqnarray*}
\frac{\partial Q_{S}}{\partial t} &=& \Delta{Q_{S}} -2{\nabla}^{i}Q_{S}{\nabla}_{i}u - 2\Big|\nabla_{i}\nabla_{j}u - S_{ij}-\frac{1}{t}g_{ij} \Big|^{2} - \Big(\frac{2}{t} -\gamma \Big)Q_{S} \\
&+& \Big(-\frac{2}{t} -\gamma \Big)|\nabla u|^{2} + 3 \gamma S +\frac{n}{t^{2}}(2-d) + d{\gamma} \frac{n}{t} - \Big(2{\mathcal H}(S_{ij}, -{\nabla}u) + {\mathcal D}(S_{ij}, -{\nabla}u) \Big) \\
&\leq & \Delta{Q_{S}} -2{\nabla}^{i}Q_{S}{\nabla}_{i}u -\frac{2}{n}\Big(\Delta u -S - \frac{n}{t} \Big)^{2} - \Big(\frac{2}{t} -\gamma \Big)Q_{S} - \frac{2}{t}|\nabla u|^{2} \\
& -& \gamma |\nabla u|^{2} + 3 \gamma S + \frac{n}{t^{2}}(2-d) + d{\gamma}\frac{n}{t} - \Big(2{\mathcal H}(S_{ij}, -{\nabla}u) + {\mathcal D}(S_{ij}, -{\nabla}u) \Big) \\
&\leq & \Delta{Q_{S}} -2{\nabla}^{i}Q_{S}{\nabla}_{i}u -\frac{2}{n}\Big(\Delta u -S - \frac{n}{t} \Big)^{2} - \Big(\frac{2}{t} -\gamma \Big)Q_{S} \\
& -& \gamma |\nabla u|^{2} + 3 \gamma S + \frac{n}{t^{2}}(2-d) + d{\gamma} \frac{n}{t} - \Big(2{\mathcal H}(S_{ij}, -{\nabla}u) + {\mathcal D}(S_{ij}, -{\nabla}u) \Big)
\end{eqnarray*}
On the other hand, we have
$$
|\nabla u|^{2} =2 \Big(\Delta u -S - \frac{n}{t} \Big) - Q_{S} -S - \frac{n}{t}(d-2).
$$
Therefore, we obtain the following:
\begin{eqnarray*}
\frac{\partial Q_{S}}{\partial t}
&\leq & \Delta{Q_{S}} -2{\nabla}^{i}Q_{S}{\nabla}_{i}u -\frac{2}{n}\Big(\Delta u -S - \frac{n}{t} \Big)^{2} - \Big(\frac{2}{t} -\gamma \Big)Q_{S} \\
& -& \gamma \Big(2 \Big(\Delta u -S - \frac{n}{t} \Big) - Q_{S} -S - \frac{n}{t}(d-2) \Big) + 3 \gamma S + \frac{n}{t^{2}}(2-d) + d{\gamma} \frac{n}{t} \\
&-& \Big(2{\mathcal H}(S_{ij}, -{\nabla}u) + {\mathcal D}(S_{ij}, -{\nabla}u) \Big) \\
&=& \Delta{Q_{S}} -2{\nabla}^{i}Q_{S}{\nabla}_{i}u -\frac{2}{n}\Big(\Delta u -S - \frac{n}{t} \Big)^{2} - \Big(\frac{2}{t} -\gamma \Big)Q_{S} \\
& -& 2 \gamma \Big(\Delta u -S - \frac{n}{t} \Big) + \gamma Q_{S} +4\gamma S + \frac{n}{t}\gamma (d-2) +  \frac{n}{t^{2}}(2-d) + d{\gamma}\frac{n}{t} \\
&-& \Big(2{\mathcal H}(S_{ij}, -{\nabla}u) + {\mathcal D}(S_{ij}, -{\nabla}u) \Big) \\
&=& \Delta{Q_{S}} -2{\nabla}^{i}Q_{S}{\nabla}_{i}u -\frac{2}{n}\Big(\Delta u -S - \frac{n}{t} \Big)^{2} - \Big(\frac{2}{t} -2\gamma \Big)Q_{S}  - 2 \gamma \Big(\Delta u -S - \frac{n}{t} \Big) \\
&+& 4\gamma S + \frac{2n}{t}\gamma (d-1) +  \frac{n}{t^{2}}(2-d) - \Big(2{\mathcal H}(S_{ij}, -{\nabla}u) + {\mathcal D}(S_{ij}, -{\nabla}u) \Big). \\
\end{eqnarray*}
Since we also have the following by a direct computation:
$$
-2 \gamma \Big(\Delta u -S - \frac{n}{t} \Big) = \frac{2}{n} \gamma \Big(\Delta u -S - \frac{n}{t} - \frac{n}{2} \Big)^{2} - \frac{2}{n}\gamma \Big(\Delta u -S - \frac{n}{t} \Big)^{2} - \frac{n}{2}\gamma.
$$
Hence, we obtain
\begin{eqnarray*}
\frac{\partial Q_{S}}{\partial t}
&\leq & \Delta{Q_{S}} -2{\nabla}^{i}Q_{S}{\nabla}_{i}u -\frac{2}{n} \Big(1 + \gamma \Big) \Big(\Delta u -S - \frac{n}{t} \Big)^{2} - \Big(\frac{2}{t} -2\gamma \Big)Q_{S}  \\
&+& \frac{2}{n} \gamma \Big(\Delta u -S - \frac{n}{t} - \frac{n}{2} \Big)^{2} + 4\gamma S + \frac{2n}{t}\gamma (d-1) +  \frac{n}{t^{2}}(2-d) - \frac{n}{2}\gamma \\
&-& \Big(2{\mathcal H}(S_{ij}, -{\nabla}u) + {\mathcal D}(S_{ij}, -{\nabla}u) \Big). \\
\end{eqnarray*}
Now suppose that $-1 \leq \gamma \leq 0$. Then the above implies
\begin{eqnarray*}
\frac{\partial Q_{S}}{\partial t}
&\leq & \Delta{Q_{S}} -2{\nabla}^{i}Q_{S}{\nabla}_{i}u - \Big(\frac{2}{t} -2\gamma \Big)Q_{S} +  4\gamma S + \frac{2n}{t}\gamma (d-1)  \\
&+&  \frac{n}{t^{2}}(2-d) - \frac{n}{2}\gamma - \Big(2{\mathcal H}(S_{ij}, -{\nabla}u) + {\mathcal D}(S_{ij}, -{\nabla}u) \Big).
\end{eqnarray*}
Assume that $d \geq 2$ holds. Moreover, by the assumption of Theorem \ref{main-Ee}, we also get
the following:
$$
2{\mathcal H}(S_{ij}, -{\nabla}u) + {\mathcal D}(S_{ij}, -{\nabla}u) \geq 0, \ S \geq 0.
$$
Then we have
\begin{eqnarray*}
\frac{\partial Q_{S}}{\partial t}
&\leq & \Delta{Q_{S}} -2{\nabla}^{i}Q_{S}{\nabla}_{i}u - \Big(\frac{2}{t} -2\gamma \Big)Q_{S} + \frac{2n}{t}\gamma (d-1) - \frac{n}{2}\gamma. \\
\end{eqnarray*}
Since we also have
$$
- \Big(\frac{2}{t} - 2\gamma \Big)Q_{S} = - \Big(\frac{2}{t} - 2\gamma \Big)\Big(Q_{S} + \frac{n}{4}\gamma \Big) - \frac{n}{2t}\gamma - \frac{n}{2}\gamma^{2},
$$
we obtain
\begin{eqnarray*}
\frac{\partial }{\partial t} \Big(Q_{S} + \frac{n}{4}\gamma \Big)
&\leq & \Delta{\Big(Q_{S} + \frac{n}{4}\gamma \Big)} -2{\nabla}^{i}\Big(Q_{S} + \frac{n}{4}\gamma \Big){\nabla}_{i}u - \Big(\frac{2}{t} -2\gamma \Big)\Big(Q_{S} + \frac{n}{4}\gamma \Big) \\
& + & \frac{n}{t}\gamma (2d-\frac{5}{2}) - \frac{n}{2}(\gamma^{2} + \gamma) \\
&\leq & \Delta{\Big(Q_{S} + \frac{n}{4}\gamma \Big)} -2{\nabla}^{i}\Big(Q_{S} + \frac{n}{4}\gamma \Big){\nabla}_{i}u - \Big(\frac{2}{t} -2\gamma \Big)\Big(Q_{S} + \frac{n}{4}\gamma \Big) \\
& - & \frac{n}{2}(\gamma^{2} + \gamma),
\end{eqnarray*}
where notice that
$$
\frac{n}{t}\gamma (2d-\frac{5}{2}) \leq 0
$$
under $-1 \leq \gamma \leq 0$ and $d \geq 2$. Finally, we get the following by taking $\gamma =-1$:
\begin{eqnarray*}
\frac{\partial }{\partial t} \Big(Q_{S} - \frac{n}{4} \Big)
&\leq & \Delta{\Big(Q_{S} - \frac{n}{4} \Big)} -2{\nabla}^{i}\Big(Q_{S} - \frac{n}{4} \Big){\nabla}_{i}u - \Big(\frac{2}{t} +2 \Big)\Big(Q_{S} - \frac{n}{4} \Big).
\end{eqnarray*}
Notice that
$$
Q_{S} < \frac{n}{4}
$$
holds for $t$ small enough which depends on $d$. By using the maximal principle, we obtain the desired result. \par

Similarly, we get the following by using Corollary \ref{cor-Da-f}:
\begin{thm}\label{main-Ee-P}
Let $g(t)$ be a solution to the geometric flow (\ref{GF}) on a closed oriented smooth $n$-manifold $M$satisfying
\begin{eqnarray*}
2{\mathcal H}(S_{ij},X) + {\mathcal D}(S_{ij}, X) \geq 0, \ S \geq 0.
\end{eqnarray*}
hold for all vector fields $X$ and all time $t \in [0, T)$ for which the flow exists. Let $v$ satisfies
\begin{eqnarray*}
\frac{\partial v}{\partial t} = \Delta v - |\nabla v|^{2} -S + \frac{n}{2t}- \Big(v +\frac{n}{2}\log(4{\pi}t) \Big).
\end{eqnarray*}
Let
\begin{eqnarray*}
{R}_{S} = 2\Delta v - |\nabla v|^{2} -3{S} -d\frac{n}{t},
\end{eqnarray*}
where $d \geq 2$ is any fixed constant. Then for all time $t \in (0, T)$,
$$
{R}_{S} \leq \frac{n}{4}
$$
holds.
\end{thm}


\section{Proof of Theorem \ref{main-F}}\label{F-proof}

By Proposition \ref{for-lem-key-1} in the case where $\alpha = 0$, $\beta = -1$, $a = c = 0$, $b = 1$, $d = 0$, we obtain
\begin{eqnarray*}
{Q}_{S} = |\nabla u|^{2} - \frac{u}{t}
\end{eqnarray*}
and
\begin{eqnarray*}
\frac{\partial Q_{S}}{\partial t}
&=& {\Delta}Q_{S} - 2{\nabla}^{i}Q_{S}{\nabla}_{i}u -2|\nabla \nabla u|^{2} - \frac{1}{t}|\nabla u|^{2} + \frac{1}{t^{2}}u  + {\mathcal D}_{(0, 0, -1)}(S_{ij}, -\nabla u) \\
&+& 2\gamma(t) |\nabla u|^{2} - \frac{\gamma(t)}{t}u \\
&=& {\Delta}Q_{S} - 2{\nabla}^{i}Q_{S}{\nabla}_{i}u - \frac{1}{t}{Q}_{S} + 2\gamma(t) |\nabla u|^{2} - \frac{\gamma(t)}{t}u + {\mathcal D}_{(0, 0, -1)}(S_{ij}, -\nabla u) \\
&=& {\Delta}Q_{S} - 2{\nabla}^{i}Q_{S}{\nabla}_{i}u - \Big(\frac{1}{t} - \gamma(t) \Big)Q_{S} + \gamma (t)|\nabla u|^{2}-2{\mathcal I}(S_{ij}, -\nabla u),
\end{eqnarray*}
where notice that ${\mathcal D}_{(0, 0, -1)}(S_{ij}, -\nabla u) = -2{\mathcal I}(S_{ij}, -\nabla u)$.
Since we assumed $\gamma (t) \leq 0$ and ${\mathcal I}(S_{ij}, -\nabla u) \geq 0$, the above implies
\begin{eqnarray*}
\frac{\partial Q_{S}}{\partial t}
&\leq& {\Delta}Q_{S} - 2{\nabla}^{i}Q_{S}{\nabla}_{i}u - \Big(\frac{1}{t} - \gamma(t) \Big)Q_{S}.
\end{eqnarray*}
Since
$$
Q_{S} < 0
$$
holds for $t$ small enough, the maximal principle tells us that the desired result holds.


\vfill

{\footnotesize
\noindent
{Hongxin Guo }\\
{School of mathematics and information science, Wenzhou University, \\
Wenzhou, Zhe-jiang 325035, China}\\
{\sc e-mail}: guo@wzu.edu.cn \\

{\footnotesize
\noindent
{Masashi Ishida}\\
{Department of Mathematics, Graduate School of Science,
Osaka University \\
1-1, Machikaneyama, Toyonaka, Osaka, 560-0043, Japan}\\
{\sc e-mail}: ishida@math.sci.osaka-u.ac.jp \\

\end{document}